\documentclass[11pt,reqno,intlimits]{amsart}
\usepackage{amsfonts,amssymb,amsmath,amsthm}
\usepackage{url}
\usepackage{graphicx}
\usepackage[colorlinks=true,urlcolor=blue,citecolor=red,linkcolor=blue,linktocpage,pdfpagelabels,bookmarksnumbered,bookmarksopen]{hyperref}
\begin{document}

\newtheorem{theorem}{Theorem}[section]
\newtheorem{definition}[theorem]{Definition}
\newtheorem{example}[theorem]{Example}
\newtheorem{proposition}[theorem]{Proposition}
\newtheorem{lemma}[theorem]{Lemma}
\newtheorem{corollary}[theorem]{Corollary}
\newtheorem{remark}[theorem]{Remark}
\renewcommand{\theequation}{\thesection.\arabic{equation}}
\newcommand{\sezione}[1]{\section{#1}\setcounter{equation}{0}}
\def\vs{\vskip0cm\noindent}
\newcommand{\red}{\textcolor{red}}
\newcommand{\cvd}{$\hfill \sqcap \hskip-6.5pt \sqcup$} 
\newcommand{\nor}{\Arrowvert}
\newcommand{\dl}{\delta_\l}
\newcommand{\utl}{\tilde{u}_\l}
\newcommand{\vtl}{\tilde{v}_\l}
\def\R{{\rm I\mskip -3.5mu R}}
\def\N{{\rm I\mskip -3.5mu N}}
\def\Z{{\rm I\mskip -3.5mu Z}}
\def\i{i^{*}}
\def\e{\varepsilon}
\def\di12{\mathcal{D}^{1,2}(\R^n)}
\def\ra{{\rightarrow}}
\def\na{\nabla}
\def\d{\delta}
\def\g{\gamma}
\def\D{\Delta}
\def\l{{\lambda}}
\def\L{{\Lambda}}
\def\a{{\alpha}}
\def\b{{\beta}}
\def\s{{\sigma}}
\def\t{{\tau}}
\def\o{\omega}
\def\il{_{i,\l}}
\def\pe{p_{\e}}
\def\kl{_{k,\l}}
\def\0l{_{0,\l}}
\def\1l{_{1,\l}}
\def\2l{_{2,\l}}
\def\3l{_{3,\l}}
\def\4l{_{4,\l}}
\def\ul{u_\l}
\def\pel{\Phi_{\e,\l}}
\def\rl{r_\l}
\def\yt{\tilde y}
\def\at{\tilde a}
\def\rt{\tilde r}
\def\gl{\gamma_\l}
\def\xl{x_\l}
\def\kx{K(x)}
\def\Il{I_\l}
\def\ue{u^\e}
\def\uedy{u^\e_{\d,y}}
\def\pedy{\phi ^\e_{\d,y}}
\def\tedy{T ^\e_{\d,y}}
\def\wl{w_\l}
\def\zl{z_\l}
\def\xil{\xi_{\l}}
\def\de{\partial}
\def\star{\frac {2n}{n+2}}
\def\crit{\frac {2n}{n-2}}
\def\sexp{\frac {ns}{n+2s}}
\def\fe{f_{\e}}
\def\fo{f_0}
\def\Om{\Omega}
\def\lra{\longrightarrow}
\def\ra{\rightarrow}
\def\mE{\mathcal{E}}

\title[On a generalized Toda System]
 {On a general $SU(3)$ Toda System}
 \thanks{The first two authors are supported by PRIN-2009-WRJ3W7 grant and
the third author is supported by NSERC of Canada}

\author[Gladiali]{Francesca  Gladiali}
\address{Francesca Gladiali, Dipartimento Polcoming, Universit\`a  di Sassari  - Via Piandanna 4, 07100 Sassari - Italy.}
\email{fgladiali@uniss.it}
\author[Grossi]{Massimo Grossi}
\address{Massimo Grossi, Dipartimento di Matematica, Universit\`a di Roma
La Sapienza, P.le A. Moro 2 - 00185 Roma- Italy.}
\email{massimo.grossi@uniroma1.it}
\author[Wei]{Juncheng Wei}
\address{Juncheng Wei, Department of Mathematics, University of British Columbia, Vancouver, BC V6T 1Z2, Canada, and Department of Mathematics, Chinese University of Hong Kong, Shatin, NT, Hong Kong}
\email{jcwei@math.ubc.ca}

\begin{abstract}
We study the  following generalized $SU(3)$ Toda System
$$
\left\{\begin{array}{ll}
-\Delta u=2e^u+\mu e^v & \hbox{ in }\R^2\\
-\Delta v=2e^v+\mu e^u & \hbox{ in }\R^2\\
\int_{\R^2}e^u<+\infty,\ \int_{\R^2}e^v<+\infty
\end{array}\right.
$$
where $\mu>-2$. We prove the existence of radial solutions bifurcating from the radial solution $(\log \frac{64}{(2+\mu) (8+|x|^2)^2}, \log \frac{64}{ (2+\mu) (8+|x|^2)^2})$ at the values $\mu=\mu_n=2\frac{2-n-n^2}{2+n+n^2},\ n\in\N $.
\end{abstract}

\maketitle

\sezione{Introduction and statement of main results}
Let us consider the following  generalized Toda system,
\begin{equation}\label{1}
\left\{\begin{array}{ll}
-\Delta u=2e^u+\mu e^v & \hbox{ in }\R^2\\
-\Delta v=2e^v+\mu e^u & \hbox{ in }\R^2\\
\int_{\R^2}e^u<+\infty,\ \int_{\R^2}e^v<+\infty
\end{array}\right.
\end{equation}
with $\mu>-2$. This system is the limiting equation for the so-called Gudnason model in non-abelian Chern-Simons theory
\begin{equation}\label{1.2n}
\left\{\begin{array}{l} \Delta U= \alpha^2 (e^{U+V}+ e^{U-V})(e^{U+V}+ e^{U-V} -2) +\alpha \beta (e^{U+V}- e^{U-V})^2 + 4\pi \sum_{j=1}^n \delta_{p_j} (x),\\
\Delta V= \alpha \beta (e^{U+V}-e^{U-V})(e^{U+V} + e^{U-V} -2) +\beta^2 (e^{2U+2V}- e^{2U-2V}) +4\pi \sum_{j=1}^n \delta_{p_j} (x)
\end{array}
\right.
\end{equation}
where $x \in \R^2, \alpha, \beta>0$ and $n$ is an nonnegative integer.  In fact, as in \cite{ALW} and \cite{CI}, we rescale (\ref{1.2n}) as $ U(x)= \tilde{U} (\epsilon x) +2\log \epsilon, V(x)= \tilde{V} (\epsilon x)$.  Letting $\epsilon \to 0$ we obtain the following system
\begin{equation}
\label{1.2m}
\left\{\begin{array}{l}
-\Delta \tilde{U}=2 \alpha^2 (e^{\tilde{U}+\tilde{V}}+ e^{\tilde{U}-\tilde{V}})+  4n\pi \delta_0 \\
-\Delta  \tilde{V}=2 \alpha \beta (e^{\tilde{U}+\tilde{V}}-e^{\tilde{U}-\tilde{V}})+ 4n \pi \delta_0
\end{array}
\right.
\end{equation}
If $n=0$, then by suitable linear transformation (\ref{1.2m}) is equivalent to (\ref{1}) with $\mu= 2 \frac{\alpha^2-\alpha\beta}{ \alpha^2+\alpha \beta}$. Clearly it holds that $\mu>-2$. 
We refer to \cite{HLTY} for the background on Gudnason model and the references therein. \\
This problem is the natural extension, in the case of systems, of the known Liouville equation
\begin{equation}\label{i1}
\left\{\begin{array}{ll}
-\Delta u=e^u & \hbox{ in }\R^2\\
\int_{\R^2}e^u<+\infty.
\end{array}\right.
\end{equation}
All solutions of \eqref{i1} have been completely classified in \cite{CL1}.\\
Contrary to the case of \eqref{i1}, it is not available a classification result for solutions of the problem \eqref{1} and the structure of the space of the solutions is far from being fully understood. \\
{ Note that if $u=v$ then system \eqref{1} reduces to the equation
\begin{equation}\nonumber
-\Delta u=(2+\mu)e^u \quad\hbox{ in }\R^2
\end{equation}
which admits, according to \cite{CL1}, the $3$-parameter family of solutions
\begin{equation}\label{3}
U_{\mu,\delta,y}(x)=\log\frac{64\delta}{(2+\mu)\left(8\delta+|x-y|^2\right)^2},
\end{equation}
with $\delta>0$ and $y\in\R^2$ and so \eqref{1} always admits the solution $(u,v)=(U_{\mu,\delta,y}, U_{\mu,\delta,y})$ for any $\mu>-2$, $\delta>0$ and $y\in\R^2$.}\\
A first general result to \eqref{1} is given in \cite{CK} where the authors consider the problem
\begin{equation}\label{i2}
\left\{\begin{array}{ll}
-\Delta u=\g_{11}e^u+\g_{12}e^v & \hbox{ in }\R^2\\
-\Delta v=\g_{21}e^u+\g_{22}e^v & \hbox{ in }\R^2\\
\int_{\R^2}e^u<+\infty,\ \int_{\R^2}e^v<+\infty.
\end{array}\right.
\end{equation}
Under some assumption on the coefficients $\g_{ij}$ including $\g_{ij}\ge0$, $$\g_{11}+\g_{12}=1$$ and $$\g_{21}+\g_{22}=1$$ the authors prove that all solutions to \eqref{i2} are radially symmetric and of the type as in \eqref{3}.\\
Always in the case where $\g_{ij}\ge0$, in \cite{CSW} some milder assumption on $\g_{ij}$ are given in order to obtain the radial symmetry of the solutions. A complete classification of radial solutions to (\ref{1}) in the case of $\mu>0$ can be found in \cite{LZ1, LZ2}.\\
On the other hand, if some of $\g_{ij}$ are negative,  we will see that the system can have {both radial and }nonradial solutions.\\
In this setting, the most studied case is that of the Cartan matrix, where $\g_{ij}$, $i,j=1,2$ are given by
\begin{equation}\label{i3}
\begin{pmatrix}
2&-1\\
-1&2
\end{pmatrix}.
\end{equation}
In this case Jost and Wang (see \cite{JW2}) provided the following beautiful classification result,\\
{\bf Theorem} (see \cite{JW2}).\\
{\em All solutions of the problem
\begin{equation}\label{i4}
\left\{\begin{array}{ll}
-\Delta u=2e^u-e^v & \hbox{ in }\R^2\\
-\Delta v=2e^v-e^u & \hbox{ in }\R^2\\
\int_{\R^2}e^u<+\infty,\ \int_{\R^2}e^v<+\infty
\end{array}\right.
\end{equation}
are given by (in complex notation)
\begin{equation}\label{i5}
\left\{\begin{array}{ll}
u(z)=\log\frac{4\left(a_1^2a_2^2+a_1^2|2z+c|^2+a_2^2|z^2+2bz+bc-d|^2\right)}
{\left(a_1^2+a_2^2|z+b|^2+|z^2+cz+d|^2\right)^2}\\
v(z)=\log\frac{16a_1^2a_2^2\left(a_1^2+a_2^2|z+b|^2+|z^2+cz+d|^2\right)}
{\left(a_1^2a_2^2+a_1^2|2z+c|^2+a_2^2|z^2+2bz+bc-d|^2\right)^2}
\end{array}\right.
\end{equation}
for some real constants $a_1,a_2>0$ and $b,c,d\in\mathbb{C}$. (See also \cite{WZZ}.) }\\[.4cm]
\noindent This shows that when some of $\g_{ij}'s$ are negative system \eqref{i2} can have {both radial and} nonradial solutions.\\
This result was extended when in \eqref{i4} appear singular sources (see\cite{LWY}).\\
The purpose of this paper is to investigate more general cases than those of the Cartan matrix and to understand for which matrices $\g_{ij}$ can be expected results of existence of solutions. So we will consider system \eqref{1} with $\mu<0$.\\
Note that variational methods seem difficult to apply because of the lack of compactness of the domain. Even the choice of a suitable functional space is quite unclear. For this reason we use the {\em bifurcation theory}.\\
{Let us denote by
\begin{equation}\label{4}
U_\mu(x)=U_{\mu,1,O}(x)=\log\frac{64}{(2+\mu)\left(8+|x|^2\right)^2}.
\end{equation}}
A first step in our proof is the {\em linearization }of the system \eqref{1} at the trivial solution $u=v=
U_\mu$. We have the following
\begin{lemma}\label{6}
Let us suppose that $(w_1,w_2)\ne(0,0)$ is a solution to
\begin{equation}\label{5}
\left\{\begin{array}{ll}
-\Delta w_1=2e^{U_\mu} w_1+\mu e^{U_\mu} w_2 & \hbox{ in }\R^2\\
-\Delta w_2=2e^{U_\mu} w_2+\mu e^{U_\mu} w_1 & \hbox{ in }\R^2
\end{array}\right.
\end{equation}
Set
\begin{equation}\label{a1}
\mu_n=2\frac{2-n-n^2}{2+n+n^2}\quad n\in\N \cup\{0\},\end{equation}
We have the following alternative,\\
$i)$ If $\mu>-2$ but $\mu\ne\mu_n$ then all the bounded solutions to \eqref{5} are given, in radial coordinates, by
\begin{equation}\label{a2}
(w_1,w_2)=\left(a_1\frac{rcos\theta}{8+r^2}+a_2\frac{r\sin\theta}{8+r^2}+a_3\frac{8-r^2}{8+r^2},a_1\frac{rcos\theta}{8+r^2}+a_2\frac{r\sin\theta}{8+r^2}+a_3\frac{8-r^2}{8+r^2}\right)
\end{equation}
for some real constants $a_1,a_2,a_3$.\\
$ii)$ if $\mu=\mu_n$ then all the bounded solutions to \eqref{5} are given, in radial coordinates, by
\begin{eqnarray}\label{a3}
&&(w_1,w_2)=\left(
\sum\limits_{m=0}^n \left[A_m\cos(m\theta)+B_m\sin(m\theta)\right]P^m_n\left(\frac{8-r^2}{8+r^2}\right)+a_1\frac{rcos\theta}{8+r^2}\right.\nonumber\\
&&\,\,\,\,\,\,\,+a_2\frac{r\sin\theta}{8+r^2}+a_3\frac{8-r^2}{8+r^2} \,\,,\,\,
-\sum\limits_{m=0}^n \left[A_m\cos(m\theta)+B_m\sin(m\theta)\right]P^m_n\left(\frac{8-r^2}{8+r^2}\right)\nonumber\\
&&\left.\,\,\,\,\,\,\,+a_1\frac{rcos\theta}{8+r^2}+a_2\frac{r\sin\theta}{8+r^2}+a_3\frac{8-r^2}{8+r^2}\right)
\end{eqnarray}
for some real constants $a_1,a_2,a_3$, $A_m$, $m=0,\dots,n$ and $B_m$, $m=1,\dots,n$, where
\begin{equation}\label{a4}
P_n^m(r)=(-1)^m(1-r^2)^\frac m2 \frac {d^m}{dr^m}P_n(r)
\end{equation}
and $P_n(r)$ are the classical Legendre polynomials defined as
\begin{equation}\label{a5}
P_n(r)=\frac1{2^nn!} \frac {d^n}{dr^n}(r^2-1)^n.
\end{equation}
\end{lemma}
\noindent From the previous lemma we see that the linearized operator \eqref{5} has two types of degeneracies. The first one, which holds for every $\mu$, is due to the invariance of the problem \eqref{1} for dilations and translations, while the second one appears only at the values $\mu_n$ defined in \eqref{a1}.\\
In Section \ref{s3} we will show that this second type of degeneracy is due to the existence of other solutions  of \eqref{1} which bifurcate from the solution $(U_{\mu},U_{\mu})$ for $\mu=\mu_n$ and $n\in \N$.\\
This is our main result,
\begin{theorem}\label{i6}
Let $n\in\N$ and $\mu_n$ be as defined in \eqref{a1}. The points $(\mu_n,U_{\mu_{n}},U_{\mu_{n}})$ are radial  bifurcation points for the curve of solutions $(\mu,U_{\mu},U_{\mu})$, i.e. for $\e>0$ small enough there exist radial solutions $(u_\e,v_\e)$ to \eqref{1} which satisfy
\begin{equation}\label{a6}
\left\{\begin{array}{ll}
u_\e(|x|)=U_{\mu_\e}+\e P_n\left(\frac{8-|x|^2}{8+|x|^2}\right)+\e(z_{1,\e}+z_{2,\e})\\
v_\e(|x|)=U_{\mu\e}+\e P_n\left(\frac{8-|x|^2}{8+|x|^2}\right)+\e(z_{1,\e}-z_{2,\e})
\end{array}\right.
\end{equation}
 with $\mu_\e$ such that $\mu_0=\mu_n$ and some functions $(z_{1,\e},z_{2,\e})$ uniformly bounded in $L^{\infty}(\R^2)$, with $z_{i,0}=0$ for $i=1,2$.
\end{theorem}
\begin{remark}
When $n\neq 1$ since $P_n(t)\neq t$ then $u_\e(|x|)\neq U_{\mu,\d,O}$ and this provides a second radial solution to the problem.
\end{remark}
\begin{remark}
If $n=1$ in Theorem \ref{i6} then $\mu_1=0$ and our system decouples in two Liouville equations. On the other hand, in this case, $P_1(t)=t$ and the solutions in \eqref{a6} coincide with $U_{\mu,\d,O}$. If $n=2$ then $\mu_2=-1$ and we fall in case of the Cartan matrix. Of course in these cases we get much weaker information of the complete classification results of \cite{CL1} and \cite{JW2}. In fact our results are significative for $n>2$.
\end{remark}
\noindent From Lemma \ref{6} we conjecture that our system admits $2n+4$ families of  solutions bifurcating from $\mu=\mu_n$ as in the case $n=2$. This conjecture is supported by the following perturbation result from $\mu_2=-1$:
\begin{theorem}
\label{thm2}
Let $ (u_{a_1, a_2}, v_{a_1, a_2})$ be the radial solution of the Toda system (\ref{i4}) given by (\ref{i5}) with $b=0, c=0, d=0$. Then there exists $\epsilon >0$ such that for any $ \mu \in (-1-\epsilon, -1+\epsilon)$ there exists a radial solution  $ (u_{\epsilon}, v_{\epsilon})$ to \eqref{1} which satisfies
\begin{equation}
u_\epsilon= u_{a_1, a_2} + \epsilon z_1, \ v_\epsilon= v_{a_1, a_2} +\epsilon z_2
\end{equation}
with some radial smooth and logrithmatically growing functions $ z_1$ and $z_2$.
\end{theorem}
\noindent
\begin{remark}
Theorem \ref{thm2} provides two families of radial solutions  to \eqref{5} such that $ u \not = v$. Even when $a_1=a_2$ the two types of solutions constructed in Theorems \ref{i6} and \ref{thm2} are different. An open question is whether or not non-radial solutions can be perturbed.
\end{remark}

\noindent The proof of Theorem \ref{i6} uses the bifurcation  result of \cite{CRJFA} (see Section \ref{s3} for the exact statement). One of the main hypothesis of this result is that the kernel of the linearized operator has to be one-dimensional. From Lemma \ref{6} we know that the linearized operator has instead a very rich kernel. Then, to overcame this problem, we consider only the case of radial solutions. In this way the kernel is reduced to be two-dimensional at the values $\mu=\mu_n$ defined in \eqref{a1}. \\
As said before, the degeneracy responsable of the bifurcation is that  generated by the radial eigenfunctions $P_n\left(\frac{8-|x|^2}{8+|x|^2}\right)$ and by the dilation invariance of the system. We solve this second problem projecting on a suitable subspace which excludes the presence of the function $\frac{8-|x|^2}{8+|x|^2}$ (the eigenfunction given by the scaling invariance of the problem).\\
There is another reason why we consider only the radial case: in fact the presence of the radial solution $ P_n\left(\frac{8-|x|^2}{8+|x|^2}\right)$ in the kernel of the linearized operator, for $\mu=\mu_n$, does not allow us to say that the solutions that we could find using other topological methods (in a nonradial setting) are nonradial.

We want to emphasize that this is a bifurcation for an operator which is strongly degenerate. The case of the bifurcation in a degenerate setting (radial degeneracy) has been considered in \cite{GGN}. In that paper the authors rule out the degeneracy of the corresponding operator considering a nondegenerate approximating problem in expanding balls and showing that the bifurcating branches of solutions converge to branches of bifurcating solutions of the original problem. Here, however, we prefer to use the projection.\\
Of course when we apply the bifurcation theory to find a solution, a Lagrange multiplier appears in the system (see Theorem  \ref{3.10}).\\
The final step of the proof  is to show that the Lagrange multiplier is $zero$. A crucial point is to prove the following {\em mass quantizaton} of the solutions $u_\e$, $v_\e$ (given in \eqref{a6}), i.e. that it holds
$$\int_{\R^2}e^{u_\e}=\int_{\R^2}e^{v_\e}=\frac{8\pi}{2+\mu_\e}.$$ 
The choice of the functional spaces plays a very important role here; we will consider some suitable weighted subspaces of $W^{2,2}\left(\R^2\right)$ already used in \cite{CI}. Moreover, to overcame some problems given by the presence of the Lagrange multiplier, we look for solutions which belong to $L^{\infty}(\R^2)$ also. (This choice is suggested by the form of the eigenfunctions given in Lemma \ref{6}). The interested reader can find all the definitions of the spaces and their main properties in Section \ref{s3}. \\
The paper is organized as follows: In Section \ref{s2} we prove Lemma \ref{6}. In Section \ref{s3} we solve system \eqref{1} up to a Lagrange multiplier and in Section \ref{s4} we prove Theorem \ref{i6} showing that the Lagrange multiplier is $zero$. Finally in Section \ref{s5} we use a perturbation approach to prove Theorem \ref{thm2}.

\sezione{The linearized operator}\label{s2}
\noindent In this Section we consider the linearized system \eqref{5} and we prove Lemma \ref{6}.
\begin{proof}[Proof of Lemma \ref{6}]
Let $w_1,w_2$ be solutions to  \eqref{5} and set $z=w_1-w_2$. We have the following alternatives:\\
$i) z\equiv0$. In this case \eqref{5} reduces to the equation
\begin{equation}\label{7}
-\Delta w_1=(2+\mu)e^{U_\mu}w_1
\end{equation}
which admits, in radial coordinates, the solutions
\begin{equation}\label{8}
w_1=w_2=\left.\sum\limits_{i=1}^2a_i\frac{\partial U_{\mu,\delta,y}}{\partial y_i}+a_3\frac{\partial U_{\mu,\delta,y}}{\partial\delta}\right|_{\delta=1,y=O}
\!\!\!\!\!\!\!\!\!\!\!\!\!\!\!\!=a_1\frac{rcos\theta}{8+r^2}+a_2\frac{r\sin\theta}{8+r^2}+a_3\frac{8-r^2}{8+r^2}
\end{equation}
for some real constants $a_1,a_2,a_3$.\\
$ii) z\not\equiv0$. Normalizing $z$ with its $||.||_\infty$ we get the $z$ is a solution to
\begin{equation}\label{9}
\left\{\begin{array}{ll}
-\Delta z=(2-\mu)e^{U_\mu}z\hbox{ in }\R^2 \\
||z||_\infty=1
\end{array}\right.
\end{equation}
which is equivalent to
\begin{equation}\label{10}
-\Delta z=\frac{2-\mu}{2+\mu}\frac{64}{\left(8+|x|^2\right)^2} z\hbox{ in }\R^2
\end{equation}
The eigenvalues of \eqref{10} (and the relative $L^{\infty}$-eigenfunctions) were computed in \cite{GG}, Theorem 11.1. So we know that \eqref{10} has a nontrivial bounded solution if and only if
\begin{equation}\label{11}
\frac{2-\mu}{2+\mu}=\frac{n(n+1)}2\quad n\in\N\cup\{0\},
\end{equation}
and
\begin{equation}\label{12}
  z(r,\theta)=\sum\limits_{m=0}^n \left[A_m\cos(m\theta)+B_m\sin(m\theta)\right]P^m_n\left(\frac{8-r^2}{8+r^2}\right), 
\end{equation}
where $P^m_n$ are the associated Legendre polynomials as defined in \eqref{a4}-\eqref{a5}. Finally each eigenvalue $c_n=\frac{n(n+1)}2$ has multiplicity $2n + 1$. From \eqref{11} we deduce that
\begin{equation}\label{13}
\mu=\mu_n=2\frac{2-n-n^2}{2+n+n^2}\quad n\in\N\cup\{0\}.
\end{equation}
Let us compute $w_1$ and $w_2$. Summing the equations in \eqref{5} we get that
\begin{equation}\label{14}
-\D(w_1+w_2)=(2+\mu)e^{U_\mu}(w_1+w_2)
\end{equation}
which implies
\begin{equation}\label{15}
w_1+w_2=\left.\sum\limits_{i=1}^2a_i\frac{\partial U_{\mu,\delta,y}}{\partial y_i} +a_3\frac{\partial U_{\mu,\delta,y}}{\partial\delta}\right|_{\delta=1,y=O}
\!\!\!\!\!\!\!\!\!\!\!\!\!\!\!\!=a_1\frac{rcos\theta}{8+r^2}+a_2\frac{r\sin\theta}{8+r^2}+a_3\frac{8-r^2}{8+r^2}.
\end{equation}
Since by \eqref{12} we derive that
\begin{equation}\label{16}
w_1-w_2= \sum\limits_{m=0}^n \left[A_m\cos(m\theta)+B_m\sin(m\theta)\right]P^m_n\left(\frac{8-r^2}{8+r^2}\right),
\end{equation}
by \eqref{15} and \eqref{16} we get \eqref{a3}.
 \end{proof}
\begin{remark}
{As said in the proof of Lemma \ref{6} the linearized operator has, for every $\mu\neq \mu_n$, a $3-$dimensional  kernel. For $\mu=\mu_n$ we have to add the $2n+1$ solutions of equation \eqref{10}
whose one radial and $2n$ nonradial. This implies that the Morse index of the solution $(U_\mu,U_\mu)$, in a suitable space, changes in $\mu=\mu_n$ by $2n+1$. The fact that the change of the Morse index in $\mu=\mu_n$ is odd is a good indication to obtain the bifurcation.}\\
\end{remark}

\sezione{An existence result}\label{s3}
\noindent From here on, all the functions that we consider will be radial, i.e. such that $w(x)=w(|x|)$ for any $x\in \R^2$.\\
Applying Lemma \ref{6} we have that \eqref{5} has, for every $\mu>-2$ and $\mu\ne\mu_n$, the solution (up to a constant)
\begin{equation}\label{n1}
(w_1,w_2)=\left(\frac{8-|x|^2}{8+|x|^2},\frac{8-|x|^2}{8+|x|^2}\right),
\end{equation}
and for  $\mu=\mu_n$ and $n\in \N$,
\begin{equation}\label{n2}
(w_1,w_2)=\left(A P_n\left(\frac{8-|x|^2}{8+|x|^2}\right)+a\frac{8-|x|^2}{8+
|x|^2},-A P_n\left(\frac{8-|x|^2}{8+|x|^2}\right) +a\frac{8-|x|^2}{8+|x|^2},\right)
\end{equation}
for some real constants $a$ and $A$, where $P_n=P_n^0$ is the $n-$th Legendre Polynomial, as defined in \eqref{a5}.\\
To obtain the bifurcation result we will use the transversality condition of Crandall and Rabinowitz (see Theorem 1.7 in \cite{CRJFA}). Although this result  is very well known, we report it for completeness. \\
Before stating it, let us introduce some notations. Let $G=G(t,x):\R\times X\rightarrow Y$ be an operator from the Banach spaces $X$ and $Y$. \\
We will denote by $G'_x$ the Fr\'echet derivative of $G$ with respect to $x$ and by $G''_{t,x}$ the Fr\'echet derivative of $G'_x$ with respect to $t$.
\begin{theorem}[\cite{CRJFA}, Theorem 1.7]\label{CR}
Let $X,Y$ be Banach spaces, $V$ a neighborhood of $0$
in $X$ and
$F:(-2,2)\times V\to Y$
have the properties:
\begin{itemize}
\item[a)]$F(t,0)=0$ for $|t|<2$,
\item[b)]The partial derivatives $F'_t$, $F'_x$ and $F''_{t,x}$ exist and are continuous,
\item[c)]$Ker(F'_x(0,0))$ and $Y\setminus R(F'_x(0,0))$ are one-dimensional,
\item[d)]$F''_{t,x}(0,0)w_0\notin R(F'_x(0,0))$, where $Ker(F'_x(0,0))=span\{w_0\}$.
\end{itemize}
If $Z$ is any complement of $Ker(F'_x(0,0))$ in $X$, then there is a neighborhood $U$
of $(0,0)$ in $\R\times X$, an interval $(-\e_0, \e_0)$, and continuous functions
$\eta:(-\e_0, \e_0)\to \R$, $z:(-\e_0, \e_0)\to Z$ such that $\eta(0)=0$, $z(0)= 0$ and
$$F^{-1}(0)\cap U=\{(\eta(\e),\e w_0+\e z(\e))\, :\,|\e|<\e_0\}\cup \{(t,0)\,:\,(t,0)\in U\}.$$
\end{theorem}
\noindent To apply the bifurcation result we have to {\em rule out} the degeneracy due to the invariance for dilations of problem \eqref{1}. To do this we ask that
the functions are orthogonal, in some sense, to the eigenfunction $(\frac{8-|x|^2}{8+|x|^2},\frac{8-|x|^2}{8+|x|^2})$. \\[.1cm]
As we did in the study of the linearized operator in Section 2,  we consider the sum and difference of the solutions $u$ and $v$ of \eqref{1}. Let 
\begin{equation}\label{n3}
\left\{\begin{array}{ll}
\tilde \phi=u+v\\
\psi=u-v.
\end{array}\right.
\end{equation}
Then problem \eqref{1} is equivalent to
\begin{equation}\label{3.1}
\left\{\begin{array}{ll}
-\Delta \tilde\phi=(2+\mu) \left(e^{\frac{\tilde \phi+\psi}2}+ e^{\frac{\tilde \phi-\psi}2}\right) & \hbox{ in }\R^2\\
-\Delta \psi=(2-\mu) \left(e^{\frac{\tilde \phi+\psi}2}- e^{\frac{\tilde \phi-\psi}2}\right) & \hbox{ in }\R^2\\
\int_{\R^2}e^{\frac{\tilde \phi+\psi}2}<+\infty,\ \int_{\R^2}e^{\frac{\tilde \phi-\psi}2}<+\infty
\end{array}\right.
\end{equation}
and $(\mu,2U_{\mu},0)$ is a curve of radial solutions of \eqref{3.1}.\\
First we shift these solutions in the origin, i.e. we let 
\begin{equation}\label{n4}
\tilde\phi=2U_{\mu}+\phi.
\end{equation}
 Then, finding solutions for \eqref{3.1} is equivalent to find solutions for the problem
\begin{equation}\label{3.2}
\left\{\begin{array}{ll}
-\Delta\phi=(2+\mu)e^{U_{\mu}} \left(e^{\frac{ \phi+\psi}2}+ e^{\frac{ \phi-\psi}2}-2\right) & \hbox{ in }\R^2\\
-\Delta \psi=(2-\mu)e^{U_{\mu}} \left(e^{\frac{ \phi+\psi}2}- e^{\frac{ \phi-\psi}2}\right) & \hbox{ in }\R^2\\
\int_{\R^2}e^{U_{\mu}}e^{\frac{ \phi+\psi}2}<+\infty,\ \int_{\R^2}e^{U_{\mu}}e^{\frac{ \phi-\psi}2}<+\infty
\end{array}\right.
\end{equation}
and $(\mu,0,0)$ is a solution of \eqref{3.2} for every value of $\mu$. Moreover $u=U_\mu+\frac 12(\phi+\psi)$ and $v=U_\mu+\frac 12(\phi-\psi)$ are solutions of \eqref{1}.\\
Let us introduce the spaces in which we apply Theorem \ref{CR}. Let $\a\in(0,1)$ be fixed, and
\begin{equation}\label{3.3}
X_{\a}=\left\{u\in L^{2}_{loc}(\R^2)\,|\, \int_{\R^2}(1+|x|^{2+\a})u^2\, dx<+\infty\right\},
\end{equation}
equipped with the inner product $(u,v)_{X_\a}=\int_{\R^2}(1+|x|^{2+\a})uv\, dx$,
\begin{equation}\label{3.4}
Y_{\a}=\left\{u\in W^{2,2}_{loc}(\R^2)\,|\, \nor \Delta u\nor^2_{X_\a}+\nor\frac u{(1+|x|^{1+\frac\a2})}\nor^2_{L^2(\R^2)}<+\infty\right\},
\end{equation}
equipped with the inner product
$$(u,v)_{Y_\a}=(\Delta u,\Delta v)_{X_\a}+\int_{\R^2}\frac {uv}{(1+|x|^{2+\a})}\, dx.$$
The spaces $X_\a$ and $Y_\a$  were introduced by Chae and Imanuvilov in \cite{CI} to find Non-Topological Multivortex Solutions for a nonlinear problem with exponential nonlinearity in $\R^2$.\\
We recall some properties of the spaces $X_\a$ and $Y_\a$ that we need in the sequel.
\begin{proposition}[\cite{CI}]\label{CIa-properties}
Let $\a\in(0,1)$, then
\begin{itemize}
\item[i)] $Y_{\a}\subset C^0_{loc}(\R^2)$ and $X_\a\subset L^1(\R^2)$.
\item[ii)] There exists $C_1>0$ such that for any $g\in Y_\a$
\begin{equation}\label{3.6}
|g(x)|\leq C_1 \nor g\nor_{Y_\a}\left( \log^+|x|+1\right)\hbox{ for any }x\in \R^2
\end{equation}
where $log^+|x|=\max\{0,\log|x|\}$.
\item[iii)] If $g\in Y_{\a}$ then $\Delta g\in X_{\a}$. 
\item[iv)] If $\Delta g=0$ we get that $g=c$ for some $c\in\R$.
\end{itemize}
\end{proposition}
Let us introduce the function
\begin{equation}\label{3.5}
M(f)(x):=-\frac 1{2\pi}\int_{\R^2}\log|x-y|f(y)\,dy.
\end{equation}
Let us recall some properties of $M$,
\begin{proposition}[\cite{CI}]\label{CIb-properties}
We have that,
\begin{itemize}
\item[i)] $M:X_{\a}\rightarrow Y_{\a}$.
\item[ii)] $-\Delta M(f)=f$ in $\R^2$.
\item[iii)] $M\in  L(X_{\a},Y_{\a})$ and
\begin{equation}\label{3.5-bis}
|M(g)(x)|\leq C_2 \nor g\nor_{X_\a}\left( \log^+|x|+1\right)\hbox{ for any }x\in \R^2,
\end{equation}
\end{itemize}
where  $C_2$ depends only on $\a$.
\end{proposition}
\noindent 
Since all eigenfunctions in \eqref{n2}  are bounded in the $L^{\infty}$-norm, it seems appropriate to use as a starting space the subset of functions of $Y_\a$
which are bounded in the $L^{\infty}$-norm. (Recall that functions in $Y_\a$ may have logaritmic growth at infinity.)
Then we define
\begin{equation}\label{nuovo-spazio-Y}
\widetilde Y_\a:=Y_\a\cap L^{\infty}(\R^2).
\end{equation}
$\widetilde Y_\a$ is a Banach space with the norm
$$\nor g\nor_{\widetilde Y_\a}=\max \{ \nor g\nor_{Y_\a}\, ,\,\nor g\nor_{\infty}\}$$
where $\nor g\nor_{\infty}$ denotes the usual norm in $L^{\infty}(\R^2)$.
\noindent Using the operator $M$ defined in \eqref{3.5},
we look for solutions of \eqref{3.2} as zeros of the operator  $\widetilde T:(-2,2)\times \widetilde Y_\a\times\widetilde Y_\a$
\begin{equation}\label{3.7}
\widetilde T(\mu,\phi,\psi)=\left(
\begin{array}{l}
\phi-M\left((2+\mu)e^{U_{\mu}} \left(e^{\frac{ \phi+\psi}2}+ e^{\frac{ \phi-\psi}2}-2\right)\right)\\
\psi-M\left((2-\mu)e^{U_{\mu}} \left(e^{\frac{ \phi+\psi}2}- e^{\frac{ \phi-\psi}2}\right)\right)
\end{array}
\right)
\end{equation}
Since $\phi$ and $\psi$ are bounded and $U_{\mu}\in X_\a$, by
Propositions \ref{CIa-properties} and \ref{CIb-properties}, the operator $\widetilde T(\mu,\phi,\psi)$ maps $(-2,2)\times  \widetilde Y_\a \times \widetilde Y_\a  \mapsto Y_\a\times Y_\a$.
Hence the zeros of $\widetilde T(\mu,\phi,\psi)$ 
 are bounded solutions of \eqref{3.2}.\\
Observe that $(\mu,0,0)$ is a curve of solutions of \eqref{3.2} in this space, i.e. $\widetilde T(\mu,0,0)=0$ for $|\mu|<2$.\\
Let us consider the linerized operator at $(\mu_n,0,0)$, i.e.
\begin{equation}\label{3.8}
\widetilde T'_{(\phi,\psi)}(\mu_n,0,0)\left(\begin{array}{l}
w_1\\
w_2
\end{array}\right)=\left(
\begin{array}{l}
w_1-M\left( (2+\mu_n)e^{U_{\mu_n}}w_1\right)\\
w_2-M\left((2-\mu_n)e^{U_{\mu_n}}w_2\right).
\end{array}
\right)
\end{equation}
As shown in the proof of Proposition \ref{6}, the solutions $w_1,w_2\in \widetilde Y_{\a}$ of
$$\widetilde T'_{(\phi,\psi)}(\mu_n,0,0)\left(\begin{array}{l}
w_1\\
w_2
\end{array}\right)=\left(\begin{array}{l} 0\\0\end{array}\right)$$
are given by
$$\left(\begin{array}{l}
w_1\\
w_2
\end{array}\right)=\left(\begin{array}{c}
a\frac{8-r^2}{8+r^2} \\
A P_n\left(a \frac{8-r^2}{8+r^2}\right)\end{array}\right)$$
for $a,A\in \R$.\\
Here we want to find a new radial solution of \eqref{3.2} which is due to the presence of the second component $w_2$ in the linearized equation (which appears only at the values $\mu=\mu_n$).
We are going to verify the assumption of the bifurcation result of Crandall and Rabinowitz, (Theorem \ref{CR}). To this end we define
\begin{equation}\label{K}
K:=\{g\in Y_\a\, :\, \int_{\R^2}(\Delta g) \frac{8-|x|^2}{8+|x|^2}\, dx=0\}.
\end{equation} 
From the definition of $Y_\a$ we get that $K$ is a linear closed subspace of $Y_{\a}$. We let $P_{K}$ be the projection of $Y_{\a}$ on $K$.
Finally we consider the operator $T(\mu,\phi,\psi)$  defined by
\begin{equation}\label{3.9}
 T(\mu,\phi,\psi)=\left(
\begin{array}{l}
 \phi -P_{K}\left(M\left((2+\mu)e^{U_{\mu}} \left(e^{\frac{ \phi+\psi}2}+ e^{\frac{ \phi-\psi}2}-2\right)\right)\right)\\
\psi-M\left((2-\mu)e^{U_{\mu}} \left(e^{\frac{ \phi+\psi}2}- e^{\frac{ \phi-\psi}2}\right)\right)
\end{array}
\right)
\end{equation}
Letting $
\widetilde K:=\left\{g\in \widetilde Y_\a\, :\, \int_{\R^2}(\Delta g) \frac{8-|x|^2}{8+|x|^2}\, dx=0\right\}\subset \widetilde Y_\a\cap K$, 
the operator  $T(\mu,\phi,\psi)$ maps $(-2,2)\times \widetilde K  \times \widetilde Y_\a $ into $K\times Y_{\a}$. \\
The zeros of the operator $T(\mu,\phi,\psi)$ then satisfy
\begin{equation}\label{3.10}
\left\{\begin{array}{ll}
-\Delta\phi=(2+\mu)e^{U_{\mu}} \left(e^{\frac{ \phi+\psi}2}+ e^{\frac{ \phi-\psi}2}-2\right)+64 L\frac{8-|x|^2}{(8+|x|^2)^3} & \hbox{ in }\R^2\\
-\Delta \psi=(2-\mu)e^{U_{\mu}} \left(e^{\frac{ \phi+\psi}2}- e^{\frac{ \phi-\psi}2}\right) & \hbox{ in }\R^2\\
\int_{\R^2}e^{U_{\mu}}e^{\frac{ \phi+\psi}2}<+\infty,\ \int_{\R^2}e^{U_{\mu}}e^{\frac{ \phi-\psi}2}<+\infty
\end{array}\right.
\end{equation}
where $L=L(\phi,\psi)\in \R$ and $\phi,\psi\in L^{\infty}$. Once we prove the existence of $\phi$ and $\psi$ verifying \eqref{3.10},
the final step will be to show that $L=0$ so that $\phi$ and $\psi$ are indeed solutions of \eqref{3.2}. This will be done in Section \ref{s4}.\\[.3cm]
To apply Theorem \ref{CR} to the operator  $T(\mu,\phi,\psi):(-2,2)\times \widetilde K  \times \widetilde Y_\a \to K\times Y_{\a}$ at the point $(\mu_n,0,0)$, we have to check hypotheses $a),\dots,d)$. 
This will be done in a series of technical lemmas.\\
Even if the statement and the proofs of these lemmas is quite long we prefer to include them in this section since this is the most delicate part of the proof.\\
We start proving the following
\begin{lemma}\label{lem3.1} We have that 
$ker\left(T'_{(\phi,\psi)}(\mu_n,0,0)\right)=span\left\{\left( \begin{array}{c} 0\\P_n(\frac{8-|x|^2}{8+|x|^2})\end{array}\right)\right\}$ and then it is one-dimensional.
\end{lemma}
\begin{proof}
Let us consider the linearized operator of $T$ in $(\mu_n,0,0)$. We have that
\begin{equation}\nonumber
T'_{(\phi,\psi)}(\mu_n,0,0)\left(\begin{array}{l}
w_1\\
w_2
\end{array}\right)=\left(
\begin{array}{l}
w_1-P_{K}\left(M\left( (2+\mu_n)e^{U_{\mu_n}}w_1\right)\right)\\
w_2-M\left((2-\mu_n)e^{U_{\mu_n}}w_2\right)
\end{array}\right)
=\left(\begin{array}{l}0\\0\end{array}\right)
\end{equation}
if and only if $(w_1,w_2)\in  \widetilde K\times \widetilde Y_{\a} $ satisfies
\begin{equation}\label{3.12-b}
\left\{\begin{array}{ll}
-\Delta w_1-\frac{64}{(8+|x|^2)^2}w_1-64L\frac{8-|x|^2}{(8+|x|^2)^3}=0 & \hbox{ in }\R^2\\
-\Delta w_2-\frac{n(n+1)}2\frac{64}{(8+|x|^2)^2}w_2=0 & \hbox{ in }\R^2
\end{array}\right.
\end{equation}
for some $L=L(w_1)\in \R$. Observe that, in the second equation we used the definition of $\mu_n$.\\
The first equation in \eqref{3.12-b} can be explicitly solved and gives the solutions
\begin{eqnarray}\label{3.12-c}
&&w_1(|x|)=C_1\frac{8-|x|^2}{8+|x|^2}+C_2\frac{|x|^2\log |x|-8\log |x|-16}{8+|x|^2}+\nonumber\\
&&\frac23L\frac{2(|x|^2-8)\log |x|-(|x|^2-8)\log(8+|x|^2)-16}{8+|x|^2},\nonumber
\end{eqnarray}
which are bounded if and only if $C_2=L=0$. Hence $w_1(|x|)=C_1\frac{8-|x|^2}{8+|x|^2}$ and again $C_1=0$ because
$w_1\in \widetilde K$.\\
The second equation, as said in Lemma \ref{6}, has the nontrivial radial bounded solution  $P_n\left(\frac{8-|x|^2}{8+|x|^2}\right)$, so that
\begin{equation}\label{ker}
Ker\left(T'_{(\phi,\psi)}(\mu_n,0,0)\right)=span <\left(\begin{array}{c}
0\\
P_n\left(\frac{8-|x|^2}{8+|x|^2}\right)
\end{array}\right)>
\end{equation}
is one dimensional.
\end{proof}
\begin{lemma}\label{lem3.2}
Let $g\in X_\a$ be a radial function. Then the ordinary differential equation
\begin{equation}\label{ordinary}
-w''-\frac 1r w'-\frac{n(n+1)}2 \frac{64}{(8+r^2)^2}w=g(r) \quad \hbox{ in }(0,+\infty)
\end{equation}
has the solution $w$ given by
\begin{equation}\label{sol}
w(r)=-P_n\left(\frac{8-r^2}{8+r^2}\right)\left\{\int_0^r\frac 1{s\left(P_n\left(\frac{8-s^2}{8+s^2}\right)\right)^2} \int_0^s tP_n\left(\frac{8-t^2}{8+t^2}\right)g(t)\, dtds +C\right\}.
\end{equation}
\end{lemma}
\begin{proof}
As explained in the proof of Lemma \ref{6} the homogeneous equation $-w''-\frac 1r w'-\frac{n(n+1)}2 \frac{64}{(8+r^2)^2}w=0$ has only the bounded solution $P_n\left(\frac{8-r^2}{8+r^2}\right)$. Then \eqref{sol} follows by a standard use of the variation of constants method. Let us show that $w$ is well defined at a zero of the Legendre function $P_n$.
Denote by $z_n\in (-1,1)$ a point where $P_n(z_n)=0$. Since the zeros of the Legendre function are {\em simple} we get $P_n(z)=\gamma_n(z-z_n)+O\left(|z-z_n|^2\right)$ in a neighborhood of $z_n$ and then it is not difficult to show that $\lim_{z\to z_n}w(z)<+\infty$. This ends the proof.
\end{proof}
\begin{lemma}\label{n5}
If $g\in X_\a$ then the function $w$ given in \eqref{sol} for $n=1$ satisfies the following asymptotic estimate
\begin{equation}\label{n6}
w(r)=-\left(\int_0^{+\infty}\frac{8-t^2}{8+t^2}tg(t)dt\right)\log r+O(1)\quad\hbox{as }r\rightarrow+\infty
\end{equation}
\end{lemma}
\begin{proof}
Note that since $g\in X_\a$ then the integral in \eqref{n6} is well defined.
Since $n=1$ the solution $w$ is given by
$$w(r)=-\frac{8-r^2}{8+r^2}\left(I(r)+C\right)$$
with
\begin{equation}\label{I(r)}
I(r)=\int_0^r\frac {(8+s^2)^2}{s(8-s^2)^2} \int_0^s t\frac{8-t^2}{8+t^2}g(t)\, dtds.
\end{equation}
Because $ \frac{8-r^2}{8+r^2}$ is bounded, we just consider the term $I(r)$ for $r$ large. Using that $ \frac {(8+s^2)^2}{s(8-s^2)^2}=\frac 1s+O\left(\frac 1{(1+s)^3}\right)$ for $s$ large enough we have that 
\begin{eqnarray}\label{n7}
&I(r)=& -\int_0^r \frac 1s\int_0^s t\frac{8-t^2}{8+t^2}g(t)\, dtds+\int_0^r O\left(\frac 1{(1+s)^3}\right) \int_0^s t\frac{8-t^2}{8+t^2}g(t)\, dtds\nonumber\\
&=&\hbox{(integrating by parts)}\nonumber\\
&=&-\log r\int_0^r t\frac{8-t^2}{8+t^2}g(t)\, dt+ \int_0^r t\log t \frac{8-t^2}{8+t^2}g(t)\, dt +\int_0^r O\left(\frac 1{(1+s)^3}\right) \int_0^s t\frac{8-t^2}{8+t^2}g(t)\, dtds\nonumber\\
&=&-\log r\int_0^r t\frac{8-t^2}{8+t^2}g(t)\, dt+I_1(r)+I_2(r)
\end{eqnarray}
Since $g\in X_\a$,
\begin{equation}\label{n8}
\int_0^r t\frac{8-t^2}{8+t^2}g(t)\, dt=\int_0^{+\infty} t\frac{8-t^2}{8+t^2}g(t)\, dt+O\left(\frac1{(1+r)^\frac\a2}\right)
\end{equation}
and
\begin{equation}\label{n9}
|I_1(r)|,|I_2(r)|=O(1).
\end{equation}
Finally by \eqref{n7}, \eqref{n8} and \eqref{n9} we have the claim.
\end{proof}
\begin{lemma}\label{n10}
If $g\in X_\a$ then the function $w$ given in \eqref{sol} for $n>1$ satisfies the following asymptotic estimate
\begin{equation}\label{n11}
w(r)=-\left(\int_0^{+\infty}P_n\left(\frac{8-t^2}{8+t^2}\right)tg(t)dt\right)\log r+O(1)\quad\hbox{as }r\rightarrow+\infty
\end{equation}
\end{lemma}
\begin{proof}
We can repeat step by step the proof of lemma \ref{n5} once we have proved the following estimate,
\begin{equation}\label{comp-asint}
 \frac {1}{s\left(P_n\left(\frac{8-s^2}{8+s^2}\right)\right)^2}=\frac 1s+O\left(\frac 1{(1+s)^3}\right)\quad \hbox{ for }s\hbox{ large enough.}
\end{equation}
First observe that $\lim_{s\to +\infty}\left(P_n\left(\frac{8-s^2}{8+s^2}\right)\right)^2=\left(P_n(-1)\right)^2=1$ from the definition of the Legendre Polynomials. Then we have that
$$\frac {1}{s\left(P_n\left(\frac{8-s^2}{8+s^2}\right)\right)^2}-\frac 1s=\frac {1-\left(P_n\left(\frac{8-s^2}{8+s^2}\right)\right)^2}{s\left(P_n\left(\frac{8-s^2}{8+s^2}\right)\right)^2}$$
and, using that $P_n$ and $P'_n$ are both bounded in $(-1,1)$, 
$$\frac {1-\left(P_n\left(\frac{8-s^2}{8+s^2}\right)\right)^2}{s}=O\left(P_n\left(\frac{8-s^2}{8+s^2}\right)P'_n\left(\frac{8-s^2}{8+s^2}\right)\left(\frac{8-s^2}{8+s^2}\right)'\right)=O\left(\frac 1{(1+s)^3}\right)$$
for $s$ large enough. 
This proves \eqref{comp-asint} and completes the proof of the lemma.
\end{proof}
\begin{corollary}\label{lem3.3}
If $g\in X_\a$ is a radial function and $w$ is the solution to \eqref{ordinary} given in \eqref{sol}, we have that $w\in L^{\infty}_{loc}([0,+\infty))$ and
\begin{itemize}
\item[i)] if $ \int_{\R^2}P_n\left(\frac{8-|x|^2}{8+|x|^2}\right) g(|x|)\, dx=0$ then $w\in L^\infty([0,+\infty))$,
\item[ii)] if $ \int_{\R^2}P_n\left(\frac{8-|x|^2}{8+|x|^2}\right) g(|x|)\, dx\ne0$ then $w$ has a logarithmic growth at infinity.
\end{itemize}
\end{corollary}
\begin{proof}
Since $g\in X_\a$ then $\int_0^st P_n\left(\frac{8-t^2}{8+t^2}\right)g(t)\, dt=O(s)$ for $s$ small enough. Then $\frac 1{s\left(P_n\left(\frac{8-s^2}{8+s^2}\right)\right)^2} \int_0^s tP_n\left(\frac{8-t^2}{8+t^2}\right)g(t)\, dt$ is uniformly bounded in a neighborhood of $0$ and so $w(r)$ is regular at the origin. Moreover, it is regular at a zero of the Legendre function $P_n$ and so it belongs to $L^{\infty}_{loc}([0,+\infty))$.
Finally $i)$ and $ii)$ follow by \eqref{n6} and \eqref{n11}.
\end{proof}
\begin{lemma}\label{N1}
We have that for any $g\in X_\a$ then $w\in Y_\a$ for any $n\ge1$.
\end{lemma}
\begin{proof}
First we observe that, by assumption $g(|x|)\in L^{2}_{loc}(\R^2)$ so that $w(|x|)\in W^{2,2}_{loc}(\R^2)$.
Next, by Corollary \ref{lem3.3} we have that there exists $C$ such that $|w(r)|\leq C(1+\log(1+r))$. 
Then
$$\int_{\R^2}\frac{w^2}{\left( 1+|x|^{1+\frac {\a}2}\right)^2}\, dx<+\infty$$
and  
$$\int_{\R^2}\left(\Delta w\right)^2\left(1+|x|^{2+\a}\right)\le C\left( \int_{\R^2}\left( g(|x|)\right)^2\left(1+|x|^{2+\a}\right)+
\int_{\R^2}\frac{w^2}{\left(8+r^2\right)^4}\right)<+\infty$$
since $g\in X_\a$. This ends the proof.\\
\end{proof}
\begin{lemma}\label{lem3.6}
Set $w\in \widetilde K$. Then we have that 
\begin{equation}\label{q-fin}
\int_{\R^2} \frac {8-|x|^2}{(8+|x|^2)^3}w \, dx=0.
\end{equation}
\end{lemma}
\begin{proof}
Let us denote by $B_R$ the ball centered at the origin and radius $R$. Integrating by parts twice we get
\begin{equation}\label{***}
\int_{B_R}\Delta w \frac {8-|x|^2}{8+|x|^2}\, dx=2\pi\left( Rw'(R)+Rw(R)\frac{16R}{(8+R^2)^2}\right)-\int_{B_R}\frac {64(8-|x|^2)}{(8+|x|^2)^3}w \, dx.
\end{equation}
We claim that there exists $R_k\to +\infty$ as $k\to +\infty$ such that
\begin{equation}\label{****}
R_kw'(R_k)\to 0.
\end{equation}
By contradiction let us assume that there exists $\a>0$ such that $|w'(R)|\geq \frac{\a}R$ for any $R\geq R_0$. This implies that $w$ is not bounded and so \eqref{****} holds.\\
Putting $R=R_k$ in \eqref{***} we obtain (since $w\in \widetilde K$) 
\begin{eqnarray}
&&0=\int_{\R^2} \Delta w \frac {8-|x|^2}{8+|x|^2}\, dx=\lim_{k\to +\infty}\int_{B_{R_k}}\Delta w \frac {8-|x|^2}{8+|x|^2}\, dx\nonumber\\
&&=\lim_{k\to +\infty}2\pi\left( R_kw'(R_k)+w(R_k)\frac{16R^2_k}{(8+R_k^2)^2}\right)-\int_{B_{R_k}}\frac {64(8-|x|^2)}{(8+|x|^2)^3}w \, dx\nonumber\\
&&=-\int_{\R^2} \frac {64(8-|x|^2)}{(8+|x|^2)^3}w \, dx\nonumber
\end{eqnarray}
and the proof is complete.
\end{proof}
\begin{lemma}\label{lem3.5}
We have that the range $R\left(T'_{(\phi,\psi)}(\mu_n,0,0)\right)\subset K\times Y_{\a}$ has codimension one and it is given by the functions $(f_1,f_2)\in K\times Y_{\a}$ that satisfies 
\begin{equation}\label{condition}
\int_{\R^2}P_n\left(\frac{8-|x|^2}{8+|x|^2}\right)\Delta f_2\, dx=0.
\end{equation} 
\end{lemma}
\begin{proof}
Let $(f_1,f_2)\in K\times Y_{\a}$. Our claim is equivalent to prove that 
\begin{equation}\label{range}
T'_{(\phi,\psi)}(\mu_n,0,0)\left(\begin{array}{l} 
w_1\\
w_2
\end{array}\right)=\left(\begin{array}{l} 
f_1\\
f_2
\end{array}\right)
\end{equation}
has a solution $(w_1,w_2)\in \widetilde K\times \widetilde Y_{\a}$ for any $f_1\in K$ and  $f_2$ satisfying \eqref{condition}.
On the other hand \eqref{range} is equivalent to find $(w_1,w_2)\in \widetilde K\times \widetilde Y_{\a}$  that solves
\begin{eqnarray}
&-\Delta w_1-(2+\mu_n)e^{U_{\mu_n}}w_1=-\Delta f_1 +L\frac {8-|x|^2}{(8+|x|^2)^3}& \hbox{ in }\R^2\label{eq-1}\\
&-\Delta w_2-(2-\mu_n)e^{U_{\mu_n}}w_2=-\Delta f_2 & \hbox{ in }\R^2\label{eq-2}
\end{eqnarray}
for any $(f_1,f_2)\in K\times Y_{\a}$ 
and for some $L=L(w_1)\in \R$. From the definition of the space $K$ and of the operator $M$  we have that 
$$L(w_1)=-\frac{\int_{\R^2}\frac {8-|x|^2}{(8+|x|^2)^3}w_1\, dx}{\int_{\R^2}\frac {(8-|x|^2)^2}{(8+|x|^2)^4}\, dx}.$$
Let us consider the function $\widetilde w_1(r)$ which solves \eqref{ordinary} with $n=1$ and $g(r)=-\Delta f_1$. Since $f_1\in K$ we get from Corollary \ref{lem3.3} and Lemma \ref{N1} that $\widetilde w_1\in \widetilde Y_\a$.\\
Finally, setting $w_1=\widetilde w_1+A\frac {8-|x|^2}{8+|x|^2}$ with $A=-\frac{\int_{\R^2}\Delta \widetilde w_1 \frac {8-|x|^2}{8+|x|^2}\, dx}{\int_{\R^2}\frac {64(8-|x|^2)^2}{(8+|x|^2)^4}\, dx}$ we get that $w_1\in \widetilde K$ and satisfies $-\Delta w_1-(2+\mu_n)e^{U_{\mu_n}}w_1=-\Delta f_1$ in $\R^2$. On the other hand, by Lemma \ref{lem3.6}, since $w_1\in \widetilde K$ we have that $\int_{\R^2} \frac {8-|x|^2}{(8+|x|^2)^3}w_1 
\, dx=0$ and then $L(w_1)=0$. So $w_1$ satisfies \eqref{eq-1}.\\
For what concerns equation \eqref{eq-2}, using again Corollary \ref{lem3.3}, we have that the solution $w_2$ belongs to $\widetilde Y_\a$ if and only if $f_2$ satisfies \eqref{condition}.
This concludes the proof. 
\end{proof}
\begin{lemma}\label{lem3.7}
We have that 
\begin{equation}\label{T''}
T''_{\mu,(\phi,\psi)}(\mu_n,0,0)\left(\begin{array}{c} 0\\
P_n\left(\frac{8-r^2}{8+r^2}\right)\end{array}\right)\notin R\left(T'_{(\phi,\psi)}(\mu_n,0,0)\right).
\end{equation}
\end{lemma}
\begin{proof}
We have that $T''_{\mu,(\phi,\psi)}(\tilde \mu, \tilde\phi,\tilde \psi)$ is given by
\begin{eqnarray}
&\left(\begin{array}{cc}
0 & 0\\
M\left(-\frac{128}{(2+\tilde \mu)^2(8+|x|^2)^2}\left(e^{\frac{ \tilde\phi+\tilde\psi}2}- e^{\frac{ \tilde\phi-\tilde\psi}2}\right)I\right) & M\left(-\frac{128}{(2+\tilde \mu)^2(8+|x|^2)^2}\left(e^{\frac{ \tilde\phi+\tilde\psi}2}+ e^{\frac{ \tilde\phi-\tilde\psi}2}\right)I\right)
\end{array}\right) &\nonumber
\end{eqnarray}
so that 
$$T''_{\mu,(\phi,\psi)}(\mu_n,0,0)\left(\begin{array}{c} 0\\
P_n\left(\frac{8-|x|^2}{8+|x|^2}\right)\end{array}\right)=\left(\begin{array}{c} 0\\
M\left(-\frac{256}{(2+\tilde \mu)^2(8+|x|^2)^2}P_n\left(\frac{8-|x|^2}{8+|x|^2}\right)\right)\end{array}\right).$$
This proves \eqref{T''} since          
\begin{eqnarray}
&&\int_{\R^2}\left(-\Delta \left(M\left(-\frac{256}{(2+\tilde \mu)^2(8+|x|^2)^2}P_n\left(\frac{8-|x|^2}{8+|x|^2}\right)\right)\right)\right)P_n\left( \frac{8-|x|^2}{8+|x|^2}\right)\, dx \nonumber\\
&&=\frac{-256}{(2+\mu_n)^2}\int_{\R^2}\frac{1}{(8+|x|^2)^2}\left(P_n\left(\frac{8-|x|^2}{8+|x|^2}\right)\right)^2\, dx\neq 0\nonumber
\end{eqnarray}
and the Range of $T'_{(\phi,\psi)}(\mu_n,0,0)$ is characterized by \eqref{condition}. 
\end{proof}
Now we have all the ingredients to apply Theorem \ref{CR} and we have the following result:
\begin{theorem}\label{t3.2}
Let us fix $n \in \N$  and let $\mu_n$ be as defined in \eqref{a1}. Then the
points $(\mu_n,0,0)$ are radial  bifurcation points for the curve of solutions $(\mu,0,0)$, i.e. for $\e>0$ small enough there exist $u_\e$, $v_\e$
satisfying
\begin{equation}\label{c1}
\left\{\begin{array}{ll}
-\Delta u_\e=2e^{u_\e}+\mu_\e e^{v_\e}+64L_\e\frac{8-|x|^2}{\left(8+|x|^2\right)^3} & \hbox{ in }\R^2\\
-\Delta v_\e=2e^{v_\e}+\mu_\e e^{u_\e}+64L_\e\frac{8-|x|^2}{\left(8+|x|^2\right)^3} & \hbox{ in }\R^2\\
\int_{\R^2}e^{u_\e}\le C,\ \int_{\R^2}e^{v_\e}\le C
\end{array}\right.
\end{equation}
for some Lagrange multiplier $L_\e$. Moreover we have that, for $\e$ small enough,
\begin{equation}\label{fin}
\left\{\begin{array}{ll}
u_\e=U_{\mu_{\e,n}}+\e P_n\left(\frac{8-|x|^2}{8+|x|^2}\right)+\e(z_{1,\e}+z_{2,\e})\\
v_\e=U_{\mu_{\e,n}}-\e P_n\left(\frac{8-|x|^2}{8+|x|^2}\right)+\e(z_{1,\e}-z_{2,\e}) 
\end{array}\right.
\end{equation}
with $\mu_\e$ such that $\mu_0=\mu_n$ and some functions $z_{1,\e},z_{2,\e}$ such that $z_{1,0}=z_{2,0}=0$ and $z_{1,\e},z_{2,\e}$ are uniformly bounded in $L^\infty(\R^2)$.
\end{theorem}
\begin{proof}
We apply Theorem \ref{CR} with $X=\widetilde K\times \widetilde Y_{\a}$, $Y=K\times Y_{\a}$ and $F=T(\mu,\phi,\psi)$ as defined in \eqref{3.9}.
We have that $a)$ clearly holds.\\
Concerning $b)$ let us consider a sequence $(\phi_k,\psi_k)\rightarrow(\phi,\psi)$ in $\widetilde K\times \widetilde Y_{\a}$. Then we have that $\phi_k\rightarrow\phi$ and $\psi_k\rightarrow\psi$ in $Y_\a$ and {\em uniformly} in $\R^2$. This is enough to conclude that $T'_\mu$, $T'_{(\phi,\psi)}$ and $T''_{\mu,(\phi,\psi)}$ are continuous. We skip the tedious details.\\
Assumption $c)$ follows by Lemma \ref{lem3.1} and Lemma \ref{lem3.5}.\\
Finally $d)$ follows by Lemma \ref{lem3.7}.\\
From Theorem \ref{CR} we have that for every $n\in \N$ there is a branch of nontrivial radial solutions $(\phi_\e,\psi_\e)$ of $T(\mu,\phi,\psi)=0$ bifurcating from $(\mu_n,0,0)$. Moreover, we know that, for $\e$ small enough,
$(\phi_\e,\psi_\e)=\left(\e z_{1,\e},\e P_n\left(\frac{8-|x|^2}{8+|x|^2}\right)+\e z_{2,\e}\right)$ with $(z_{1,\e},z_{2,\e})$ belonging to a neighborhood of the origin in $\widetilde K \times \widetilde Y_\a $ such that $z_{i,0}=0$ for $i=1,2$.
Then $(\phi_\e,\psi_\e)$  satisfy, for $\e$ small enough,
\begin{equation}
\left\{\begin{array}{ll}
\phi_\e=\e z_{1,\e} \nonumber\\
\psi_\e=\e P_n\left(\frac{8-|x|^2}{8+|x|^2}\right)+\e z_{2,\e}
\end{array}\right.
\end{equation}
with $z_{1,\e},z_{2,\e}$ uniformly bounded in $L^\infty(\R^2)$. Recalling the definition of $\phi$ and $\psi$ in \eqref{n3},  \eqref{n4} we get the claim.
\end{proof}
\sezione{The Lagrange multiplier is zero}\label{s4}
\noindent In this section we will show that the Lagrange multiplier $L_\e$ appearing in  \eqref{c1} is zero.

In the sequel we denote by $C$ a generic constant (independent of $n$) which can change from line to line.\\
Let us start this section proving a bound on $L_\e$.
\begin{lemma}\label{c2}
We have that
\begin{equation}\label{d2}
|L_\e|\le C.
\end{equation}
\end{lemma}
\begin{proof}
Let us multiply the first equation of \eqref{c1} by $\frac{8-|x|^2}{8+|x|^2}$ and integrate over $\R^2$. We get
$$64 L_\e\int_{\R^2}\frac{\left(8-|x|^2\right)^2}{\left(8+|x|^2\right)^4}\, dx=\int_{\R^2}\left(-\Delta u_\e\right)\frac{8-|x|^2}{8+|x|^2}\, dx-\int_{\R^2}\left(2e^{u_\e}+\mu_{\e,n} e^{v_\e}\right)\frac{8-|x|^2}{8+|x|^2}\, dx.$$
We know that $||u_\e||_{Y_\a}\leq ||U_{\mu_{\e,n}} ||_{Y_\a}+C\leq C$ 
since  $\mu_{\e,n}>-2+\delta$ for $\e$ small enough. Then we have that
\begin{equation}
\int_{\R^2}\left(-\Delta u_\e\right)\frac{8-|x|^2}{8+|x|^2}\leq||u_\e||_{Y_\a}\left(\int_{\R^2}\frac 1{1+|x|^{2+\a}}\left(\frac{8-|x|^2}{8+|x|^2}\right)^2\right)^\frac12\leq C.
\end{equation}
Then the uniform bound on $\int_{\R^2}e^{u_\e}\,,\int_{\R^2}e^{v_\e}$ in \eqref{c1} gives that
\begin{equation}
|L_\e|\le C
\end{equation}
which gives the claim.
\end{proof}
\noindent Now we prove some useful estimates.
\begin{proposition}\label{c2-bis}
If $u_\e$ and $v_\e$ satisfy \eqref{c1} we have that
\begin{equation}\label{c3}
\left\{\begin{array}{ll}
u_\e(x)=-\frac1{2\pi}\left(2\int_{\mathbb{R}^2}e^{u_\e}+\mu_{\e,n}\int_{\mathbb{R}^2}e^{v_\e}\right)\log|x|+O(1)
\\ &\hbox{as }|x|\rightarrow+\infty\\
v_\e(x)=-\frac1{2\pi}\left(2\int_{\mathbb{R}^2}e^{v_\e}+\mu_{\e,n}\int_{\mathbb{R}^2}e^{u_\e}\right)\log|x|+O(1)
\end{array}\right.
\end{equation}
\end{proposition}
\begin{proof}
Recalling that $u_\e$ and $v_\e$ are radial, we write the ODE corresponding to \eqref{c1},
\begin{equation}\label{n13}
\left\{\begin{array}{ll}
- u_\e''-\frac1r u_\e'=2e^{u_\e}+\mu_{\e,n} e^{v_\e}+64L_\e
\frac{8-r^2}{\left(8+r^2\right)^3} & \hbox{ in }(0,+\infty)\\
- v_\e''-\frac1r v_\e'=2e^{v_\e}+\mu_{\e,n} e^{v_\e}+64L_\e
\frac{8-r^2}{\left(8+r^2\right)^3} & \hbox{ in }(0,+\infty)
\end{array}\right.
\end{equation}
Let us prove our estimates for $u_\e$ (the case of $v_\e$ is the same). Integrating  \eqref{n13} twice  we get
\begin{eqnarray}\label{n14}
&&-u_\e(r)+u_\e(0)=\int_0^r\left(2e^{u_\e}+\mu_{\e,n} e^{v_\e}\right)sds\log r-
\int_0^r\left(2e^{u_\e}+\mu_{\e,n} e^{v_\e}\right)s\log sds\nonumber\\
&&+2L_\e\frac {r^2}{8+r^2}.
\end{eqnarray}
Then, from \eqref{fin} we deduce that
\begin{equation}\label{n20}
e^{u_\e},e^{v_\e}\le \frac C{1+|x|^4}
\end{equation}
So from  \eqref{n14} we derive
\begin{equation}\label{n21}
-u_\e(r)+u_\e(0)=\left(\int_0^{+\infty}\left(2e^{u_\e}+\mu_{\e,n} e^{v_\e}\right)sds\right)\log r+O(1),
\end{equation}
which gives the claim.
\end{proof}
Next proposition states the important {\em mass quantizaton} of the solutions $u_\e$ and $v_\e$,
\begin{proposition}\label{c4}
We have that
\begin{equation}\label{c5}
\int_{\mathbb{R}^2}e^{u_\e}=\int_{\mathbb{R}^2}e^{v_\e}=\frac{8\pi}{2+\mu_{\e,n}}
\end{equation}
\end{proposition}
\begin{proof}
Using \eqref{fin} and the $L^\infty$-boundedness of $z_1$ and $z_2$ we have that $\left|u_\e(x)+4\log x\right|\le C $ and $\left|v_\e(x)+4\log x\right|\le  C \hbox{ if }|x|\ge R$ for some  $R\in \R$ sufficiently large.
By Proposition \ref{c2-bis} we deduce that
\begin{equation}
\left\{\begin{array}{ll}
-\frac1{2\pi}\left(2\int_{\mathbb{R}^2}e^{u_\e}+\mu_{\e,n}\int_{\mathbb{R}^2}e^{v_\e}\right)+4=0
\\
-\frac1{2\pi}\left(2\int_{\mathbb{R}^2}e^{v_\e}+\mu_{\e,n}\int_{\mathbb{R}^2}e^{u_\e}\right)+4=0
\end{array}\right.
\end{equation}
which gives the claim.
\end{proof}

\begin{proposition}\label{c6}
Let us suppose that $u$ and $v$ are smooth solutions of
\begin{equation}\label{c7}
\left\{\begin{array}{ll}
-\Delta u=2e^u+\mu e^v+H(x) & \hbox{ in }\R^2\\
-\Delta v=2e^v+\mu e^u+H(x) & \hbox{ in }\R^2\\
\end{array}\right.
\end{equation}
with $H\in C^0(\R^2)$.
Then we have the following Pohozaev identity,
\begin{eqnarray}\label{c8}
&&R\mu\int_{\partial B_R}\nabla u\cdot\nabla v-2R\mu\int_{\partial B_R}\frac{\partial u}{\partial\nu}\frac{\partial v}{\partial\nu}-R\int_{\partial B_R}\left(|\nabla u|^2+|\nabla v|^2 \right)+\nonumber\\
&&2R\int_{\partial B_R}\left(\left(\frac{\partial u}{\partial\nu} \right)^2+\left(\frac{\partial v}{\partial\nu} \right)^2\right)+R(4-\mu^2)\int_{\partial B_R}\left(e^u+e^v\right)=
\nonumber\\
&&
2(4-\mu^2)\int_{B_R}\left(e^u+e^v\right)+(\mu-2)\int_{B_R}H(x)(x\cdot\nabla u+x\cdot\nabla v)
\end{eqnarray}
\end{proposition}
\begin{proof}
It is the same as in \cite{CK}(see also \cite{JW1})
\end{proof}
\begin{corollary}\label{cor9}
Assume that $H(x)$ is a radial function and that that $u=u(r)$ and $v=v(r)$ are radial solutions of \eqref{c7}.
Let us suppose there exists $\tilde R$ such that $\int_0^{\tilde R} H(r)rdr=0$. Then we have that
\begin{eqnarray}\label{c10}
&&(4-\mu^2)\left[\left(\int_0^{\tilde R}e^urdr\right)^2+\left(\int_0^{\tilde R} e^vrdr\right)^2+\mu\left(\int_0^{\tilde R} e^urdr\right)\left(\int_0^{\tilde R} e^vrdr\right)
\right]+\nonumber\\
&&(4-\mu^2)\tilde{R}^2\left(e^{u(\tilde R)}+e^{v(\tilde R)}\right)=
2(4-\mu^2)\left(\int_0^ {\tilde R}e^urdr+\int_0^ {\tilde R}e^vrdr\right)+\nonumber\\
&&(\mu-2)\int_0^ {\tilde R}H(r)(u'+v')r^2dr.
\end{eqnarray}
\end{corollary}
\begin{proof}
Using that $u$ and $v$ are radial solutions of \eqref{c7}, then \eqref{c8} becomes
\begin{eqnarray}\label{c7-bis}
&&-\mu R^2 u'(R)v'(R)+R^2\left(u'(R)^2+v'(R)^2\right)+(4-\mu^2)R^2\left(e^{u(R)}+e^{v(R)}\right)=\nonumber\\
&&2(4-\mu^2)\int_0^R\left(e^u+e^v\right)rdr+(\mu-2)\int_0^RH(r)(u'+v')r^2dr
\end{eqnarray}
for any $R$.
Integrating  \eqref{c7} and recalling that $\int_0^{\tilde R} H(r)rdr=0$  we get that
\begin{equation}\label{c10-bis}
\left\{\begin{array}{ll}
- \tilde Ru'(\tilde R)=2\int_0^{\tilde R}e^urdr+\mu\int_0^{\tilde R}e^vrdr\\
- \tilde Rv'(\tilde R)=2\int_0^{\tilde R} e^vrdr+\mu\int_0^{\tilde R}e^urdr\\
\end{array}\right.
\end{equation}
Setting $\alpha=\int_0^{\tilde R} e^urdr$, $\beta=\int_0^{\tilde R} e^vrdr$ and recalling \eqref{c5} we have that  \eqref{c7-bis} becomes
\begin{eqnarray}\label{c9}
&&-\mu(2\a+\mu\b)(2\b+\mu\a)+(2\a+\mu\b)^2+(2\b+\mu\a)^2+(4-\mu^2)\tilde R^2\left(e^{u(\tilde R)}+e^{v(\tilde R)}\right)
=\nonumber\\
&& 2(4-\mu^2)(\a+\b)+(\mu-2)\int_0^{\tilde R}H(r)(u'+v')r^2dr,
\end{eqnarray}
which gives the claim.
\end{proof}

\begin{proposition}\label{c11} Let $u_\e$ and $v_\e$ be the solutions of \eqref{c1}.
We have that $L_\e=0$ in \eqref{c1} for $\e$ small enough.
\end{proposition}
\begin{proof}
From \eqref{fin}
we deduce that
\begin{equation}\label{c12}
R^2\left(e^{u_\e(R)}+e^{v_\e(R)}\right)\rightarrow0\hbox{ as }R\rightarrow+\infty.
\end{equation}
Moreover, let us observe that integrating \eqref{c1} between $0$ and $r$ we have, using \eqref{c5},\\
\begin{equation}\label{c15}
|u_\e'(r)r|\le 4+CL_\e
\end{equation}
and since by Lemma \ref{c2} $L_\e$ is uniformly bounded, we derive that $u_\e'(r)$ is bounded for $r$ large enough.\\
Then, we apply the Pohozaev identity \eqref{c8} with $H(x)=64L_\e\frac{8-|x|^2}{\left(8+|x|^2\right)^3}$ as in \eqref{c1} and $\mu=\mu_{\e,n}$. Using that
$\int_0^{+\infty}  \frac{8-r^2}{\left(8+r^2\right)^3}r dr=0$ and that $e^{u_\e}$, $e^{v_\e}$ and $\frac{8-r^2}{\left(8+r^2\right)^3}(u_\e'+v_\e')r^2$ belong to $L^1(\R^2)$,
we can apply \eqref{c7-bis} with $\tilde R=+\infty$ getting that
\begin{eqnarray}\label{c13}
&&(4-\mu_{\e,n}^2)\left[\left(\int_0^\infty e^{u_\e}rdr\right)^2+\left(\int_0^\infty e^{v_\e}rdr\right)^2+\mu_{\e,n}^2\left(\int_0^\infty e^{u_\e}rdr\right)\left(\int_0^\infty e^{v_\e}rdr\right)
\right]=\nonumber\\
&&2(4-\mu_{\e,n}^2)\int_0^\infty \left(e^{u_\e}+e^{v_\e}\right) r dr + 64(\mu_{\e,n}^2-2)L_\e\int_0^\infty\frac{8-r^2}{\left(8+r^2\right)^3}(u_\e'+v_\e')r^2dr.
\end{eqnarray}
Then using the mass quantization property of Proposition \ref{c4} we have that $\int_0^\infty e^{u_\e}rdr=\int_0^\infty e^{v_\e}rdr=\frac{4}{2+\mu_{\e,n}}$. So \eqref{c13} becomes
\begin{equation}\label{c14}
0=L_\e\int_0^\infty\frac{8-r^2}{\left(8+r^2\right)^3}(u_\e'+v'_\e)r^2dr.
\end{equation}

In the previous section we showed that $u_\e,v_\e\rightarrow U_{\mu_n}$ as $\e\rightarrow0$. This implies that $u_\e',v_\e'\rightarrow-\frac{4r}{8+r^2}$ pointwise in $\R$.\\
Passing to the limit as $\e\rightarrow0$ in \eqref{c14} we get that $L_\e$=0 for $\e$ small.
This ends the proof.
\end{proof}
{\bf Proof of Theorem \ref{i6}}\\
It follows by Theorem \ref{t3.2} and Proposition \ref{c11}.

\sezione{A perturbation result: proof of Theorem \ref{thm2}}\label{s5}

The proof of Theorem \ref{thm2} follows simply from a perturbation argument. To this end, we  fix $a_1>0, a_2>0$ and $\mu= -1 + \delta$. Denote $ (U_0, V_0)= (u_{a_1, a_2},v_{a_1, a_2})$. We write $ (u, v)= ( U_0, V_0) + (\phi, \psi)$ and substitute into the system (\ref{1}) to obtain
\begin{equation}
\label{nm1}
\left\{\begin{array}{l}
\Delta \phi + 2e^{U_0} \phi - e^{V_0} \psi = -\mu (e^{V_0+\psi} -e^{V_0} -e^{V_0} \psi ) -(\mu+1) e^{V_0}  -2 (e^{U_0+\phi}-e^{U_0}- e^{U_0}\phi) \\
\Delta \psi + 2e^{V_0} \psi - e^{U_0} \phi = -\mu (e^{U_0+\phi} -e^{U_0} -e^{U_0} \phi ) -(\mu+1) e^{U_0}  -2 (e^{V_0+\phi}-e^{V_0}- e^{V_0}\phi) \\
\phi= \phi (r),\psi= \psi (r)
\end{array}
\right.
\end{equation}

Here we work in the radial class.  As in Section \ref{s2} we consider the invertibility of the linearized operator in some suitable Sobolev spaces.  To this end, recall $X_\alpha$ and $ Y_\alpha$ introduced in Section \ref{s2}. We denote $ X_{\alpha, r}$ and $Y_{\alpha, r}$ as the radial functions in $X_\alpha$ and $Y_\alpha$ respectively.

On $X_{\alpha, r}$ and $Y_{\alpha,r}$, we equip with two norms respectively:
\begin{equation}
\label{norm12}
\| f \|_{**}= \sup_{y \in \R^2} (1+|y|)^{2+\alpha}|f(y)|, \ \ \| h \|_{*}=\max (\| h\|_{Y_\alpha}, \sup_{y \in \R^2} (\log (2+|y|))^{-1} |h(y)|).
\end{equation}

The proof of Theorem \ref{thm2} follows from the following invertibility of the linearized operator (\ref{nm1}) in the radial class and contraction mapping theorem.

\begin{lemma}
\label{lemma2}
Assume that $h=(h_1, h_2) \in X_{\alpha, r}$. Then  one can find a unique solution $ (\phi, \psi)=T^{-1}(h) \in Y_{\alpha, r}$ such that
\begin{equation}\label{e305}
\left\{\begin{array}{c}
\Delta \phi+ 2 e^{U_0}\phi- e^{V_0}\psi=h_1\\
\Delta \psi+ 2 e^{V_0} \psi - e^{U_0} \phi=h_2\\
\phi= \phi (r), \ \  \psi =\psi (r)
\end{array},
\right. \,
\| \phi \|_{*} +\| \psi\|_{*} \leq C ( \| h_1 \|_{**} + \| h_2 \|_{**}) .
\end{equation}
Moreover, the map $ (h_1, h_2) \stackrel{T} {\longrightarrow} (\phi, \psi)$ can be made continuous and smooth.
\end{lemma}

\begin{proof} This follows directly from Lemma 2.3 of \cite{ALW}  by restricting  to the radial class. We omit the details.
\end{proof}

\end{document}